\documentclass[a4paper]{amsart}
\usepackage{amsmath,amsthm,amssymb}
\usepackage[british]{babel}
\usepackage[pdfborder={0 0 0}]{hyperref}
\usepackage{enumerate}
\usepackage{mathrsfs}
\usepackage{thmtools}
\usepackage{microtype}
\usepackage{graphicx}

\newtheorem{thm}{Theorem}[section]
\newtheorem{cor}[thm]{Corollary}
\newtheorem{lem}[thm]{Lemma}
\newtheorem{prop}[thm]{Proposition}
\theoremstyle{definition}
\newtheorem{defi}[thm]{Definition}
\newtheorem{ex}[thm]{Example}

\DeclareFontFamily{T1}{pzc}{} 
\DeclareFontShape{T1}{pzc}{m}{it}{<-> s * [1.15] pzcmi8t}{} 
\DeclareMathAlphabet{\mathpzc}{T1}{pzc}{m}{it} 

\newcommand{\A}{\mathcal{A}}
\newcommand{\B}{\mathcal{B}}
\newcommand{\Th}{\mathrm{Th}}

\newcommand{\M}{\mathpzc{M}}
\newcommand{\C}{\mathcal{C}}
\newcommand{\BB}{\mathscr{B}}
\newcommand{\BC}{\mathscr{C}}
\newcommand{\D}{\mathcal{D}}
\newcommand{\E}{\mathcal{E}}

\renewcommand{\phi}{\varphi}

\newcommand{\conc}{%
  \mathord{
    \mathchoice
    {\raisebox{1ex}{\scalebox{.7}{$\frown$}}}
    {\raisebox{1ex}{\scalebox{.7}{$\frown$}}}
    {\raisebox{.7ex}{\scalebox{.5}{$\frown$}}}
    {\raisebox{.7ex}{\scalebox{.5}{$\frown$}}}
  }
}

\begin{document}

\title{Natural Factors of the Medvedev Lattice Capturing IPC}
\author[R. Kuyper]{Rutger Kuyper}
\address[Rutger Kuyper]{Radboud University Nijmegen\\
Department of Mathematics\\
P.O.\ Box 9010, 6500 GL Nijmegen, the Netherlands.}
\email{r.kuyper@math.ru.nl}
\thanks{Research supported by NWO/DIAMANT grant 613.009.011 and by 
John Templeton Foundation grant 15619: `Mind, Mechanism and Mathematics: Turing Centenary Research Project'.}
\date{\today}
\maketitle

\begin{abstract}
Skvortsova showed that there is a factor of the Medvedev lattice which captures intuitionistic propositional logic (IPC). However, her factor is unnatural in the sense that it is constructed in an ad hoc manner. We present a more natural example of such a factor.
We also show that the theory of every non-trivial factor of the Medvedev lattice is contained in Jankov's logic, the deductive closure of IPC plus the weak law of the excluded middle $\neg p \vee \neg\neg p$. This answers a question by Sorbi and Terwijn.
\end{abstract}

\section{Introduction}

The Brouwer--Heyting--Kolmogorov interpretation for intuitionistic logic gives an informal relation between proofs and constructions. Since computations are a special kind of construction, it therefore seems reasonable to suspect that there is also a relation between constructive proofs and computations.
There are several approaches to making such a connection in a mathematically rigorous way. Probably the best known of these is Kleene realisability \cite{kleene-1945}, which turns out to correspond to a proper extension of intuitionistic logic. Both Kleene realisability and variants of it have been well-studied, see e.g.\ van Oosten \cite{vanoosten-2008}.

Medvedev \cite{medvedev-1955} followed an alternative path, in an attempt to formalise Kol\-mo\-go\-rov's calculus of problems. He introduced the \emph{Medvedev lattice} $\M$, which is a lattice arising from computability-theoretic considerations. Furthermore, it is a Brouwer algebra and therefore provides a semantics for an intermediate propositional logic, i.e.\ a propositional logic lying between intuitionistic propositional logic (IPC) and classical logic. Unfortunately, this approach also turns out to capture a proper extension of IPC: namely, IPC plus the weak law of the excluded middle $\neg p \vee \neg\neg p$. The same holds for the closely related \emph{Muchnik lattice} $\M_w$, which was introduced by Muchnik in \cite{muchnik-1963}.

However, this does not mean it is impossible to capture IPC using the Medvedev lattice. For any Brouwer algebra $\BB$ and any $x \in \BB$, the factor $\BB / \{y \in \BB \mid y \geq x\}$ (which we will denote by $\BB / x$) is also a Brouwer algebra. Thus one might ask if the next-best thing holds for the Medvedev lattice: is there an $\A \in \M$ such that the theory of $\M / \A$ is exactly $\mathrm{IPC}$? Quite impressively, Skvortsova \cite{skvortsova-1988-en} showed that there is such a principal factor of the Medvedev lattice which captures IPC. Unfortunately, the class $\A$ generating this factor is unnatural in the sense that it is constructed in an ad hoc manner. This leads to the natural question, posed in Terwijn \cite{terwijn-2006}: are there any natural principal factors of the Medvedev lattice which have IPC as their theory?

For the Muchnik lattice one can ask a similar question. Sorbi and Terwijn \cite{sorbi-terwijn-2012} showed that there is also a principal factor of the Muchnik lattice with IPC as its theory, but it suffers from the same problem as Skvortsova's factor of the Medvedev lattice. In \cite{kuyper-2013-2}, the author has shown that there are natural principal factors of the Muchnik lattice which capture IPC. These factors are defined using common notions from computability theory, such as lowness, 1-genericity, hyperimmune-freeness and computable traceability.

In this paper we present progress towards an affirmative answer to the question formulated above, by showing that there are principal factors of the Medvedev lattice capturing IPC which are more natural than the one given by Skvortsova.  These factors arise from the computability-theoretic notion of a \emph{computably independent set}: that is, a set $A$ such that for every $i \in \omega$ we have that $\bigoplus_{j \neq i} A^{[j]} \not\geq_T A^{[i]}$, where $A^{[i]}$ is the $i^\textrm{th}$ column of $A$, i.e.\ $A^{[i]}(n) = A(\langle i,n \rangle)$ . We can now state the main theorem of this paper.

\begin{restatable}{thm}{mainthm}
\label{thm-main}
Let $A$ be a computably independent set. Then
\[\mathrm{Th}\left(\M / \left\{i \conc f \mid f \geq_T A^{[i]}\right\}\right) = \mathrm{IPC}.\]
\end{restatable}

The existence of computably independent sets was first proven by Kleene and Post \cite{kleene-post-1954}. In fact, almost all sets are computably independent: both in the measure-theoretic sense, because every $1$-random is computably independent by van Lambalgen's theorem (see e.g.\ Downey and Hirschfeldt \cite[Theorem 6.9.1]{downey-hirschfeldt-2010}), and also in the Baire category sense, because every $1$-generic is computably independent by the genericity analogue of van Lambalgen's theorem (see e.g.\ \cite[Theorem 8.20.1]{downey-hirschfeldt-2010}).

We note that the factor from Theorem \ref{thm-main} is not nearly as natural as the factors for the Muchnik lattice from \cite{kuyper-2013-2}, where for example it is shown that
\[\mathrm{Th}\left(\M_w / \left\{f \mid f \text{ is not low}\right\}\right) = \mathrm{IPC}.\]
(Note that this factor does not work for the Medvedev lattice by \cite[p.\ 138]{skvortsova-1988-en}.)
On the other hand, the factor from Theorem \ref{thm-main} is far more natural than the one given by Skvortsova: our factor is easily definable from just a computably independent set, which occurs naturally in computability theory. Furthermore, while Skvortsova used a deep result by Lachlan, we manage to work around this and therefore our proof is more elementary.

We also study a question posed by Sorbi and Terwijn in \cite{sorbi-terwijn-2008}. As mentioned above, the theory of the Medvedev lattice is equal to Jankov's logic $\mathrm{Jan}$, the deductive closure of $\mathrm{IPC}$ plus the weak law of the excluded middle $\neg p \vee \neg\neg p$. Let $0'$ be the mass problem consisting of all non-computable functions. Recall that we say that a mass problem is \emph{Muchnik} if it is upwards closed under Turing reducibility. In \cite{sorbi-terwijn-2008} it is shown that for all Muchnik $\B >_\M 0'$ the theory of the factor $\M / \B$ is contained in $\mathrm{Jan}$. Therefore, Sorbi and Terwijn asked: is $\mathrm{Th}(\M / \B)$ contained in $\mathrm{Jan}$ for all mass problems $\B >_\M 0'$?

Sorbi and Terwijn also proposed a connected question: does every $\B >_\M 0'$ bound a join-irreducible Medvedev degree $>_\M 0'$? By their results, this would imply that $\mathrm{Th}(\M / \B)$ is always contained in $\mathrm{Jan}$. However, they conjectured the answer to this connected question to be negative, a fact which was later proven by Shafer \cite{shafer-2011}. Nonetheless, in the same paper, Shafer widened the class of mass problems $\B$ for which $\mathrm{Th}(\M / \B) \subseteq \mathrm{Jan}$ holds to those $\B$ which bound a `pseudo-meet' of a countable sequence of join-irreducible degrees. Unfortunately, Shafer also showed that this still does not cover all $\B >_\M 0'$.

We give a positive answer to Sorbi and Terwijn's question. This is accomplished by showing that a relativisation of Theorem \ref{thm-main} holds, i.e.\ that for every $\B > 0'$ there is in fact a factor $\C \leq_\M \B$ such that $\mathrm{Th}(\M / \C) = \mathrm{IPC}$.

Our notation is mostly standard. We denote the natural numbers by $\omega$, Cantor space by $2^\omega$ and Baire space by $\omega^\omega$. For any set $X \subseteq \omega^\omega$ we denote by $C(X)$ the upper cone $\{f \in \omega^\omega \mid \exists g \in X (f \geq_T g)\}$. By $\conc$ we denote concatenation of strings. For any set $\A \subseteq \omega^\omega$ we denote by $\overline{\A}$ its complement in $\omega^\omega$. For unexplained notions from computability theory, we refer to Odifreddi \cite{odifreddi-1989}, for the Muchnik and Medvedev lattices, we refer to the surveys of Sorbi \cite{sorbi-1996} and Hinman \cite{hinman-2012} (but we use the notation from Sorbi and Terwijn \cite{sorbi-terwijn-2008}), and finally for unexplained notions from lattice theory we refer to Balbes and Dwinger \cite{balbes-dwinger-1975}.

\section{Preliminaries}

First, let us recall the definition of the Medvedev lattice.

\begin{defi}{\rm (Medvedev \cite{medvedev-1955})}
Let $\A,\B \subseteq \omega^\omega$ (we will call subsets of $\omega^\omega$ \emph{mass problems}). 
We say that $\A$ \emph{Medvedev reduces to} $\B$ (denoted by $\A \leq_\M \B$) if there exists a Turing functional $\Phi$ such that $\Phi(\B) \subseteq \A$. If both $\A \leq_\M \B$ and $\B \leq_\M \A$ we say that $\A$ and $\B$ are \emph{Medvedev equivalent} (denoted by $\A \equiv_\M \B$). The equivalence classes of mass problems under Medvedev equivalence are called \emph{Medvedev degrees}, and the class of all Medvedev degrees is denoted by $\M$.
\end{defi}

Instead of the usual notation $\vee$ for joins (least upper bounds) and $\wedge$ for meets (greatest lower bounds) in lattices, we use $\oplus$ respectively $\otimes$. The reason for this is that we will shortly see that $\oplus$ corresponds to logical conjunction $\wedge$, while $\otimes$ corresponds to logical disjunction $\vee$.

\begin{defi}{\rm(McKinsey and Tarski \cite{mckinsey-tarski-1946})}
A \emph{Brouwer algebra} is a bounded distributive lattice together with a binary \emph{implication operator} $\to$ satisfying:
\[a \oplus c \geq b \text{ if and only if } c \geq a \to b\]
i.e.\ $a \to b$ is the least element $c$ satisfying $a \oplus c \geq b$.
\end{defi}

As the name suggests, the Medvedev lattice is a lattice. In fact, it is also a Brouwer algebra, as the next proposition shows.

\begin{prop}{\rm (\cite{medvedev-1955})}
The Medvedev lattice is a Brouwer algebra under the operations induced by:
\begin{align*}
\A \oplus \B &= \{f \oplus g \mid f \in \A \text{ and } g \in \B\}\\
\A \otimes \B &= \{0 \conc f \mid f \in \A\} \cup \{1 \conc g \mid g \in \B\}\\
\A \to \B &= \{n \conc f \mid \forall g \in \A (\Phi_n(f \oplus g) \in \B\}.
\end{align*}
Furthermore, the bottom element $0$ is the Medvedev degree of $\omega^\omega$, while the top element $1$ is the Medvedev degree of $\emptyset$.
\end{prop}

The main reason Brouwer algebras are interesting is because we can use them to give algebraic semantics for IPC, as witnessed by the next definition and the results following after it.

\begin{defi}{\rm(\cite{mckinsey-tarski-1948})}
Let $\phi(x_1,\dots,x_n)$ be a propositional formula with free variables among $x_1,\dots,x_n$, let $\BB$ be a Brouwer algebra and let $b_1,\dots,b_n \in \BB$. Let $\psi$ be the formula in the language of Brouwer algebras obtained from $\phi$ by replacing logical disjunction $\vee$ by $\otimes$, logical conjunction $\wedge$ by $\oplus$, logical implication $\to$ by Brouwer implication $\to$ and the false formula $\bot$ by $1$ (we view negation $\neg\alpha$ as $\alpha \to \bot$). We say that $\phi(b_1,\dots,b_n)$ \emph{holds in $\BB$} if $\psi(b_1,\dots,b_n) = 0$. Furthermore, we define the \emph{theory} of $\BB$ (notation: $\Th(\BB)$) to be the set of those formulas which hold for every valuation, i.e.\
\[\Th(\BB) = \{\phi(x_1,\dots,x_m) \mid \forall b_1,\dots,b_m \in \BB(\phi(b_1,\dots,b_m) \text{ holds in } \BB)\}.\]
\end{defi}

The following soundness result is well-known and directly follows from the observation that all rules in some fixed deduction system for IPC preserve truth.

\begin{prop}{\rm(\cite[Theorem 4.1]{mckinsey-tarski-1948})}
For every Brouwer algebra $\BB$: $\mathrm{IPC} \subseteq \Th(\BB)$.
\end{prop}
\begin{proof}
See e.g.\ Chagrov and Zakharyaschev \cite[Theorem 7.10]{chagrov-zakharyaschev-1997}.
\end{proof}

Conversely, the class of Brouwer algebras is complete for $\mathrm{IPC}$.

\begin{thm}{\cite[Theorem 4.3]{mckinsey-tarski-1948})}
\[\bigcap\{\mathrm{Th}(\BB) \mid \BB \text{ a Brouwer algebra}\}  = \mathrm{IPC}\]
\end{thm}

Thus, Brouwer algebras can be used to provide algebraic semantics for $\mathrm{IPC}$. Therefore, it would be nice if the computationally motivated Medvedev lattice has $\mathrm{IPC}$ as its theory, so that it would provide computational semantics for $\mathrm{IPC}$. Unfortunately the weak law of the excluded middle holds in the Medvedev lattice, as can be easily verified. However, as mentioned in the introduction we can still recover $\mathrm{IPC}$ by looking at principal factors of the Medvedev lattice.

\begin{prop}
Let $\BB$ be a Brouwer algebra and let $x,y \in \BB$. Then the interval $[x,y]_\BB = \{z \in \BB \mid x \leq z \leq y\}$ is a sublattice of $\BB$. Furthermore, it is a Brouwer algebra under the implication
\[u \to_{[x,y]_\BB} v = (u \to_\BB v) \oplus x.\]
\end{prop}

\begin{prop}
Let $\BB$ be a Brouwer algebra and let $x \in \BB$. Then $\BB / \{z \in \BB \mid z \geq x\}$, which we will denote by $\BB / x$, is isomorphic as a bounded distributive lattice to $[0,x]_\BB$. In particular, $\BB / x$ is a Brouwer algebra.
\end{prop}

Thus, looking at a principal factor $\M / \A$ is the same as restricting to $[\omega^\omega,\A]_{\M}$. This means that, when looking at the theory of this factor, we interpret $\bot$ as $\A$ instead of as $\emptyset$. So one might interpret looking at such a factor by replacing the problem $\emptyset$, which is `too hard', by an easier problem $\A$.

Finally, we mention one easy lemma which we will use in this paper.

\begin{lem}\label{surj-th}
Let $\BB, \BC$ be Brouwer algebras and let $\alpha: \BB \to \BC$ be a surjective homomorphism. Then $\mathrm{Th}(\BB) \subseteq \mathrm{Th}(\BC)$.
\end{lem}
\begin{proof}
Let $\phi(x_1,\dots,x_n) \not\in \mathrm{Th}(\BC)$. Fix $c_1,\dots,c_n \in \BC$ such that $\phi(c_1,\dots,c_n) \not= 0$. Fix $b_1,\dots,b_n \in \BB$ such that $\gamma(b_i) = c_i$. Then
\[\alpha(\phi(b_1,\dots,b_n)) = \phi(\alpha(b_1),\dots,\alpha(b_n)) = \phi(c_1,\dots,c_n) \not=0\]
because $\alpha$ is a homomorphism. Thus $\phi(b_1,\dots,b_n) \not= 0$ and therefore 
$\phi \not\in \mathrm{Th}(\BB)$.
\end{proof}

\section{Upper implicative semilattice embeddings of $\mathcal{P}(I)$ into $\M$}

As a first step, we will describe a method to embed Boolean algebras of the form $\mathcal{P}(I)$, ordered under reverse inclusion $\supseteq$, into the Medvedev lattice $\M$ as an upper implicative semilattice (i.e.\ preserving $\oplus$, $\to$, $0$ and $1$). It should be noted that we will only need this for finite $I$, and Skvortsova \cite[Lemma 7]{skvortsova-1988-en} already showed that such embeddings exist. However, Skvortsova used Lachlan's result \cite{lachlan-1968} that every countable distributive lattice can be order-theoretically embedded as an initial segment of the Turing degrees. Because we want natural factors of the Medvedev lattice, we want to avoid the use of this theorem. Our main result of this section will show that there are various natural embeddings of $\mathcal{P}(I)$ into $\M$.
These embeddings are induced by so-called \emph{strong upwards antichains}, where the notion of a strong upwards antichain is the order-dual of the notion of an antichain normally used in forcing.

\begin{defi}
Let $\A \subseteq \omega^\omega$ be downwards closed under Turing reducibility and let $(f_i)_{i \in I} \in \A^I$. Then we say that $(f_i)_{i \in I}$ is a \emph{strong upwards antichain in $\A$} if for all $i \neq j$ we have that $f_i \oplus f_j \not\in \A$.
\end{defi}

Henceforth we will mean by \emph{antichain} a \emph{strong upwards antichain}.

\begin{ex}
We give some examples of countably infinite antichains.
\begin{enumerate}[{\rm (i)}]
\item Take $\A$ to be the computable functions together with the functions of minimal degree, and $f_0,f_1 \dots$ any sequence of functions of distinct minimal Turing degree.
\item Let $f_0,f_1,\dots$ be pairwise incomparable under Turing reducibility and take $\A$ to be the lower cone of $\{f_i \mid i \in \omega\}$.
\end{enumerate}
\end{ex}

The next theorem shows that each antichain induces an upper implicative semilattice embedding of $\mathcal{P}(I)$ in a natural way.

\begin{thm}\label{thm-pow-embed}
Let $\A \subseteq \omega^\omega$ be downwards closed under Turing reducibility, let $(f_i)_{i \in I}$ be an antichain in $\A$, and let $\B = \overline{\A} \cup C\left(\{f_i \mid i \in I\}\right)$.
Then the map $\alpha$ given by $\alpha(X) = \overline{\A} \cup C\left(\{f_i \mid i \in X\}\right)$
is an upper implicative semilattice embedding of $(\mathcal{P}(I),\supseteq)$ into $\Big[\B,\overline{\A}\Big]_\M$.
\end{thm}
\begin{proof}
For ease of notation, if $X \subseteq I$ we will denote by $C(X)$ the set $C\left(\{f_i \mid i \in X\}\right)$.

We have:
\[\alpha(X \cap Y) = \overline{\A} \cup C(X \cap Y).\]
On the other hand, because $\alpha(X)$ and $\alpha(Y)$ are upwards closed their join is just intersection (see Skvortsova \cite[Lemma 5]{skvortsova-1988-en}),
and therefore:
\[\alpha(X) \oplus \alpha(Y) \equiv_\M \overline{\A} \cup (C(X) \cap C(Y)).\]
Clearly, $\alpha(X \cap Y) \subseteq \overline{\A} \cup (C(X) \cap C(Y))$. Conversely, let $g \in \overline{\A} \cup (C(X) \cap C(Y))$. If $g \not\in \A$ then clearly $g \in \alpha(X \cap Y)$. So, assume $g \in \A$. Let $i \in X, j \in Y$ be such that $g \geq_T f_i$ and $g \geq_T f_j$. Then $f_i \oplus f_j \leq_T g \in \A$ so $f_i \oplus f_j \in \A$. Since $(f_i)_{i \in I}$ is an antichain in $\A$ this can only be the case if $i = j$, so we see that $g \in \alpha(X \cap Y)$.

We also have, again by \cite[Lemma 5]{skvortsova-1988-en}:
\begin{align*}
\alpha(X) &\to_{\big[\B,\overline{\A}\big]_\M} \alpha(Y)\\
&\equiv_\M \B \oplus \{g \mid \forall h \in \alpha(X) (g \oplus h \in \alpha(Y))\}\\
&\equiv_\M \{g \in \B \mid \forall i \in X \forall h \geq_T f_i \exists j \in Y (g \oplus h \in \A \to g \oplus h \geq_T f_j)\}\\
&= \overline{\A} \cup \{g \in C\left(\{f_i \mid i \in I\}\right)\notag\\
&\quad\quad\quad\quad\quad\mid \forall i \in X \forall h \geq_T f_i \exists j \in Y (g \oplus h \in \A \to g \oplus h \geq_T f_j)\}.
\end{align*}
Fix any $g \in \A \cap C\left(\{f_i \mid i \in I\}\right)$ such that
\begin{equation}\label{eqn1}
\forall i \in X \forall h \geq_T f_i \exists j \in Y (g \oplus h \in \A \to g \oplus h \geq_T f_j).
\end{equation}
Then we know that there is some $k \in I$ such that $g \geq_T f_k$. We claim: $k \not\in X$ or $k \in Y$.

Namely, assume $k \in X$ and $k \not\in Y$. Then, by \eqref{eqn1} (with $h=g$) there exists some $j \in Y$ such that $g \geq_T f_j$, and since $k \not\in Y$ we know that $j \neq k$. But then $f_k \oplus f_j \leq_T g \in \A$ so $f_k \oplus f_j \in \A$, a contradiction with the fact that $(f_i)_{i \in I}$ is an antichain in $\A$.

Conversely, if $g \in \A$ is such that $g \geq_T f_k$ for some $k \not\in X$ or some $k \in Y$, then \eqref{eqn1} holds: namely, if $k \not\in X$ then we have for all $i \in X$ that $g \oplus f_i \not\in \A$ because $(f_i)_{i \in I}$ is an antichain in $\A$, while if $k \in Y$ we have that $g \oplus f_i \geq_T f_k$.

So, from this we see:
\begin{align*}
\alpha(X) \to_{\big[\B,\overline{\A}\big]_\M} \alpha(Y)
&\equiv_\M \overline{\A} \cup C((I \setminus X) \cup Y)\\
&= \alpha((I \setminus X) \cup Y)\\
&= \alpha(X \to_{\mathcal{P}(I)} Y).\qedhere
\end{align*}
\end{proof}

\section{From embeddings of $\mathcal{P}(\omega)$ to factors capturing IPC}\label{sec-to-factors}

In this section we will show how to construct a more natural factor of the Medvedev lattice with IPC as its theory; that is, we will prove Theorem \ref{thm-main}. For this proof we will use several ideas from Skvortsova's construction of a factor of the Medvedev lattice which has IPC as its theory, given in Skvortsova \cite{skvortsova-1988-en}. We combine these ideas with our own to get to the factor in Theorem \ref{thm-main}. First, let us discuss canonical subsets of a Brouwer algebra.

\begin{defi}{\rm (\cite[p.\ 134]{skvortsova-1988-en})}
Let $\BB$ be a Brouwer algebra and let $\BC \subseteq \BB$. Then we call $\BC$ \emph{canonical} if:
\begin{enumerate}[\rm (i)]
\item \label{canon-1}All elements in $\BC$ are meet-irreducible,
\item \label{canon-2}$\BC$ is closed under joins and implications (i.e.\ it is a sub-upper implicative semilattice),
\item \label{canon-3}For all $a \in \BC$ and $b,c \in \B$ we have $a \to (b \otimes c) = (a \to b) \otimes (a \to c)$.
\end{enumerate}
\end{defi}

\begin{prop}{\rm (\cite[Corollary to Lemma 6]{skvortsova-1988-en})}\label{prop-muchnik-can}
The set of Muchnik degrees is a canonical subset of $\M$.
\end{prop}

\begin{cor}\label{cor-range-canonical}
The range of $\alpha$ from Theorem \ref{thm-pow-embed} is canonical in $[\alpha(I),\alpha(\emptyset)]_\M$.
\end{cor}
\begin{proof}
The range of $\alpha$ consists of Muchnik degrees, so \eqref{canon-1} holds by Proposition \ref{prop-muchnik-can}. Furthermore, $\alpha$ is an upper implicative semilattice embedding, and therefore \eqref{canon-2} also holds. Finally, if $\C_0,\C_1 \in [\alpha(I),\alpha(\emptyset)]_\M$ and $X \subseteq I$, then we see, using Proposition \ref{prop-muchnik-can}:
\begin{align*}
\alpha(X) &\to_{[\alpha(I),\alpha(\emptyset)]_\M} (\C_0 \otimes \C_1)\\
&= (\alpha(X) \to_\M (\C_0 \otimes \C_1)) \oplus \alpha(I)\\
&\equiv_\M ((\alpha(X) \to_\M \C_0) \otimes (\alpha(X) \to_\M \C_1)) \oplus \alpha(I)\\
&\equiv_\M (\alpha(X) \to_{[\alpha(I),\alpha(\emptyset)]_\M} \C_0) \otimes (\alpha(X) \to_{[\alpha(I),\alpha(\emptyset)]_\M} \C_1).\qedhere
\end{align*}
\end{proof}

\begin{prop}{\rm (\cite[Lemma 2]{skvortsova-1988-en})}\label{prop-free-canonical}
If $\BC$ is a canonical set in a Brouwer algebra $\BB$, then the smallest sub-Brouwer algebra of $\BB$ containing $\BC$ is $\{a_1 \otimes \dots \otimes a_n \mid a_i \in \BC\}$, and it is isomorphic to the free Brouwer algebra over the upper implicative semilattice $\BC$ through an isomorphism fixing $\BC$.
\end{prop}

In particular, we see:

\begin{cor}\label{cor-emb-ext}
If we let $\alpha$ be the embedding of $(\mathcal{P}(I),\supseteq)$ from Theorem \ref{thm-pow-embed}, then $\{\alpha(X_1) \otimes \dots \otimes \alpha(X_n) \mid X_i \in \mathcal{P}(I)\}$ is a sub-Brouwer algebra of $[\alpha(I),\alpha(\emptyset)]_\M$ which is isomorphic to the free Brouwer algebra over the upper implicative semilattice $(\mathcal{P}(I),\supseteq)$.
\end{cor}
\begin{proof}
From Corollary \ref{cor-range-canonical} and Proposition \ref{prop-free-canonical}.
\end{proof}

Let $\BB_n$ be the Brouwer algebra of the upwards closed subsets of $(\mathcal{P}(\{1,\dots,n\}) \setminus \{\emptyset\},\supseteq)$ ordered under reverse inclusion $\supseteq$, i.e.\ the elements of $\BB_n$ are those $A \subseteq \mathcal{P}(\{1,\dots,n\}) \setminus \emptyset$ such that if $X \in A$ and $Y \in \mathcal{P}(\{1,\dots,n\}) \setminus \{\emptyset\}$ is such that $X \supseteq Y$, then $Y \in A$. We can use $\BB_n$ to capture IPC in the following way:

\begin{prop}{\rm (\cite[the remark following Lemma 3]{skvortsova-1988-en})}\label{prop-free-n-ipc}
$\bigcap_{n > 0} \bigcap_{x \in \BB_n} \mathrm{Th}(\BB_n / x) = \mathrm{IPC}$.
\end{prop}
\begin{proof}
Let $\mathrm{LM} = \bigcap_{n > 0} \mathrm{Th}(\BB_n)$, the \emph{Medvedev logic of finite problems}. Given a set of formulas $X$, let $X^+$ denote the set of positive (i.e.\ negation-free) formulas in $X$. Then $\mathrm{LM}^+ = \mathrm{IPC}^+$, see Medvedev \cite{medvedev-1962}.

Now, let $\phi(x_1,\dots,x_m)$ be any formula. Let $\phi'(x_1,\dots,x_{m+1})$ be the formula where $x_{m+1}$ is a fresh variable and where $\bot$ is replaced by $x_1 \wedge \dots \wedge x_{m+1}$, so $\phi'$ is negation-free. Then, if $\phi \not\in \mathrm{IPC}$, we have $\phi' \not\in \mathrm{IPC}^+$ (see Jankov \cite{jankov-1968-2}), so there are $n \in \omega$ and $x_1,\dots,x_{m+1} \in \BB_n$ such that $\phi'(x_1,\dots,x_{m+1}) \not= 0$. Let $x = x_1 \oplus \dots \oplus x_{m+1}$, then $\phi \not\in \mathrm{Th}(\BB_n / x)$.
\end{proof}

Furthermore, it is easy to obtain these $\BB_n$ as free distributive lattices over upper implicative semilattices, as expressed by the following proposition.

\begin{prop}{\rm (\cite[Lemma 3]{skvortsova-1988-en})}\label{prop-free-n-isom}
The Brouwer algebra $\BB_n$ is isomorphic to the free distributive lattice over the upper implicative semilattice $(\mathcal{P}(\{1,\dots,n\}),\supseteq)$.
\end{prop}

\begin{cor}\label{cor-n-isom}
Let $I$ be a set of size $n$. 
If we let $\alpha$ be the embedding of $(\mathcal{P}(I),\supseteq)$ from Theorem \ref{thm-pow-embed}, then $\{\alpha(X_1) \otimes \dots \otimes \alpha(X_m) \mid m \in \omega \wedge \forall i \leq m (X_i \in \mathcal{P}(I))\}$ is a sub-Brouwer algebra of $[\alpha(I),\alpha(\emptyset)]_\M$ isomorphic to $\BB_n$.
\end{cor}
\begin{proof}
From Corollary \ref{cor-emb-ext} and Proposition \ref{prop-free-n-isom}.
\end{proof}

The following lemma allows us to compare the theories of different intervals.

\begin{lem}{\rm (\cite[Lemma 4]{skvortsova-1988-en})}\label{lem-intervals}
In any Brouwer algebra $\BB$: if $x,y,z \in \BB$ are such that $x \oplus z = y$, then $\mathrm{Th}([0,z]_\BB) \subseteq \mathrm{Th}([x,y]_\BB)$.
\end{lem}
\begin{proof}
Let $\gamma: [0,z]_\BB \to [x,y]_\BB$ be given by $\gamma(u) = x \oplus u$. This map is well-defined, since if $u \leq z$, then
$x \oplus u \leq x \oplus z = y$. Clearly $\gamma$ preserves $\oplus$ and $\otimes$, while for $\to$ we have:
\[\gamma(u \to_{[0,z]_\BB} v) = (u \to_\BB v) \oplus x = ((u \oplus x) \to_\BB (v \oplus x)) \oplus x = \gamma(u) \to_{[x,y]_\BB} \gamma(v).\]
Furthermore, $\gamma$ is surjective, so the result now follows from Lemma \ref{surj-th}.
\end{proof}

Before we get to the proof of Theorem \ref{thm-main} we need one theorem from computability theory.

\begin{thm}\label{thm-splitting-ext}
Let $A,E \in 2^\omega$ be such that $E \geq_T A'$. Let $B_0,B_1,\dots \in 2^\omega$ be uniformly computable in $E$ and such that $A \not\geq_T B_i$ . Then there exists a set $D \geq_T A$ such that $D' \leq_T E$ and such that for all $i \in \omega$ we have $D \oplus B_i \geq_T E$.
\end{thm}
\begin{proof}
This follows from relativising Posner and Robinson \cite[Theorem 3]{posner-robinson-1981} to $A$.
\end{proof}

Finally, we need an easy lemma on extending computably independent sets. For ease of notation, let us assume that our pairing function is such that $(A \oplus B)^{[2i]} = A^{[i]}$ and $(A \oplus B)^{[2i+1]} = B^{[i]}$.

\begin{lem}\label{lem-extend}
Let $A$ be a computably independent set. Then there exists a set $B$ such that $A \oplus B$ is computably independent.
\end{lem}
\begin{proof}
Our requirements are as follows:
\begin{align*}
R_{\langle e,2n \rangle}&: A^{[n]} \not= \{e\}^{\bigoplus_{i \not= 2n} (A \oplus B)^{[i]}}\\
R_{\langle e,2n+1 \rangle}&: B^{[n]} \not= \{e\}^{\bigoplus_{i \not= 2n+1} (A \oplus B)^{[i]}}.
\end{align*}

We build $B$ by the finite extension method, i.e.\ we define strings $\sigma_0 \subseteq \sigma_1 \subseteq \dots$ and let $B = \bigcup_{s \in \omega} \sigma_s$. For ease of notation, define $\sigma_{-1} = \emptyset$. At stage $s$, we deal with requirement $R_s$. There are two cases:
\begin{itemize}
\item $s = \langle e,2n \rangle$: if there is a string $\sigma$ extending $\sigma_{s-1}$ and an $m \in \omega$ such that $\{e\}^{\bigoplus_{i \not= 2n} (A \oplus \sigma)^{[i]}}(m){\downarrow} \not= A^{[n]}(m)$, take $\sigma_s$ to be the least such $\sigma$. Otherwise, let $\sigma_s = \sigma_{s-1}$.
\item $s = \langle e,2n+1 \rangle$: if there exists a string $\sigma$ extending $\sigma_{s-1}$ such that we have $\{e\}^{\bigoplus_{i \not= 2n+1} (A \oplus \sigma)^{[i]}}(|\sigma_{s-1}| + 1){\downarrow}$, take the least such $\sigma$ and let $\sigma_s$ be the least string extending $\sigma_{s-1}$ which coincides with $\sigma$ outside the $n^\mathrm{th}$ column and such that $\sigma_s^{[n]}(|\sigma_{s-1}| + 1) = 1 - \{e\}^{\bigoplus_{i \not= 2n+1} (A \oplus \sigma)^{[i]}}(|\sigma_{s-1}| + 1)$. Otherwise, let $\sigma_s = \sigma_{s-1}$.
\end{itemize}

We claim: $B$ is as required. To this end, we verify the requirements:
\begin{itemize}
\item $R_{\langle e,2n \rangle}$: towards a contradiction, assume $A^{[n]} = \{e\}^{\bigoplus_{i \not= 2n} (A \oplus B)^{[i]}}$. Let $s = \langle e,2n \rangle$.
By construction we then know for every $\sigma$ extending $\sigma_{s-1}$ and every $m \in \omega$ that, if $\{e\}^{\bigoplus_{i \not= 2n} (A \oplus \sigma)^{[i]}}(m){\downarrow}$, we have $\{e\}^{\bigoplus_{i \not= 2n} (A \oplus \sigma)^{[i]}}(m) = A^{[n]}(m)$. Furthermore, for every $m \in \omega$ there is a string $\sigma$ extending $\sigma_{s-1}$ such that $\{e\}^{\bigoplus_{i \not= 2n} (A \oplus \sigma)^{[i]}}(m){\downarrow}$: just take a suitably long initial segment of $B$. However, this means that $\bigoplus_{i \not= n} A^{[i]} \geq_T A^{[n]}$, which contradicts $A$ being computably independent.
\item $R_{\langle e,2n+1 \rangle}$: let $s = \langle e,2n+1 \rangle$. Then by our construction we know that, if $\{e\}^{\bigoplus_{i \not= 2n+1} (A \oplus B)^{[i]}}(|\sigma_{s-1}| + 1) \downarrow$, then it differs from $B^{[n]}(|\sigma_{s-1}| + 1)$.
\end{itemize}
\end{proof}

We can now prove Theorem \ref{thm-main}.

\mainthm*
\begin{proof}
Fix $n \in \omega$ and $x \in \BB_n$. Let $I = \{1,\dots,n\}$. For now assume we have some downwards closed $\A$ and an antichain $D_1,\dots,D_n \in \A$. Then 
Corollary \ref{cor-n-isom} tells us that
\[\{\alpha(Y_1) \otimes \dots \otimes \alpha(Y_m) \mid m \in \omega \wedge \forall i \leq m(Y_i \in \mathcal{P}(I))\}\]
is a subalgebra of $\Big[\overline{\A} \cup C(\{D_1,\dots,D_n\}),\overline{\A}\Big]_\M$ isomorphic to $\BB_n$.
So, there are $X_1,\dots,X_k \subseteq I$ such that we can embed $\BB_n / x$ as subalgebra of
\[\Big[\overline{\A} \cup C(\{D_1,\dots,D_n\}), \alpha(X_1) \otimes \dots \otimes \alpha(X_k) \Big]_\M.\]
If we would additionally have that 
\begin{equation}\label{thm-main-eqn1}
\Big(\overline{\A} \cup C(\{D_1,\dots,D_n\})\Big) \oplus \left\{i \conc f \mid f \geq_T A^{[i]}\right\} \equiv_\M \alpha(X_1) \otimes \dots \otimes \alpha(X_k),
\end{equation}
then Lemma \ref{lem-intervals} tells us that
\begin{align*}
\mathrm{Th}\Big(\M / \Big\{i \conc f &\mid f \geq_T A^{[i]}\Big\}\Big)\\
&\subseteq \mathrm{Th}\Big(\Big[\overline{\A} \cup C(\{D_1,\dots,D_n\}),\alpha(X_1) \otimes \dots \otimes \alpha(X_k)\Big]_\M\Big)\\
&\subseteq \mathrm{Th}(\BB_n / x).
\end{align*}
Now, if we would be able to do this for arbitrary $n \in \omega$ and $x \in \BB_n$, then Proposition \ref{prop-free-n-ipc} tells us that
\[\mathrm{Th}\left(\M / \left\{i \conc f \mid f \geq_T A^{[i]}\right\}\right) = \mathrm{IPC},\]
so then we would be done.

\bigskip
Thus, it suffices to show that for all $n \in \omega$ and all $X_1,\dots,X_k \subseteq \{1,\dots,n\}$ there exists a downwards closed $\A$ and an antichain $D_1,\dots,D_n \in \A$ such that \eqref{thm-main-eqn1} holds. Fix a $B$ for $A$ as in Lemma \ref{lem-extend}. Let $\A = \omega^\omega \setminus C(\{(A \oplus B)'\})$. For every $1 \leq i \leq n$ fix a $D_i \geq_T \left(\bigoplus_{1 \leq j \leq k, i \in X_j} A^{[j]}\right) \oplus B^{[i]}$ such that $D_i' \leq_T (A \oplus B)'$, such that $D_i \oplus A^{[j]} \geq_T (A \oplus B)'$ for every $j \in \{1 \leq j \leq k \mid i \not\in X_j\} \cup \{k+1,k+2,\dots\}$ and such that $D_i \oplus B^{[j]} \geq_T (A \oplus B)'$ for every $j \not= i$, which exists by Theorem \ref{thm-splitting-ext}.

We claim: $\{D_1,\dots,D_n\}$ is an antichain in $\A$. Clearly, $D_1,\dots,D_n \in \A$. Next, let $1 \leq i < j \leq n$. Then:
\[D_i \oplus D_j \geq_T B^{[i]} \oplus D_j \geq_T (A \oplus B)',\]
so $D_i \oplus D_j \not\in \A$.

Thus, we need to show that \eqref{thm-main-eqn1} holds.
First, let $g \in \overline{\A} \cup C(\{D_1,\dots,D_n\})$ and let $f \geq_T A^{[j]}$. If $j > k$, then either $g \geq_T D_i$ for some $1 \leq i \leq n$ and $f \oplus g \geq_T A^{[j]} \oplus D_i \geq_T (A \oplus B)'$, or $g \geq_T (A \oplus B)'$ and then also $f \oplus g \geq_T (A \oplus B)'$. In both cases we see that
$f \oplus g \in \overline{\A} \subseteq \alpha(X_1)$.

Thus, we may assume that $j \leq k$. We claim: $f \oplus g \in \alpha(X_j)$. Indeed, if $g \geq_T D_i$ for some $i \in X_j$, then $f \oplus g \geq_T D_i$ and $C(D_i) \subseteq \alpha(X_j)$, while if $g \geq_T D_i$ for some $i \not\in X_j$, then $f \oplus g \geq_T A^{[j]} \oplus D_i  \geq_T (A \oplus B)'$, and finally, if $g \geq_T (A \oplus B)'$ then clearly $f \oplus g \geq_T (A \oplus B)'$. Thus, we see that $f \oplus g$ computes an element of $\alpha(X_1) \otimes \dots \otimes \alpha(X_k)$, and that this computation is in fact uniform in $(j \conc f) \oplus g$.

For the other direction, note that for fixed $1 \leq i \leq k$ we have that
\[C(X_i) \subseteq C(\{D_1,\dots,D_n\})\]
and also that
\[C(X_i) \subseteq \left\{f \mid f \geq_T A^{[i]}\right\}\]
because for every $j \in X_i$ we have that $D_j \geq_T A^{[i]}$.
\end{proof}

\section{Relativising the construction}

We will next show that Skvortsova's construction can be performed below every mass problem $\B >_\M 0'$. This also implies that for every $\B >_\M 0'$ we have that $\mathrm{Th}(\M / \B) \subseteq \mathrm{Jan}$, answering a question by Sorbi and Terwijn; see Corollary \ref{cor-jan} below.

First, note that for every $\B > 0'$ we can find a countable mass problem $\E \subseteq 0'$ such that $\E \not\geq_\M \B$ (e.g.\ by taking one function for every $n \in \omega$ witnessing that $\Phi_n(0') \not\subseteq \B$). Then the set $\{A \mid \forall f \in \E (A \not\geq_T f)\}$ has measure $1$ (by Sack's result that upper cones in the Turing degrees have measure $0$, see e.g.\ Downey and Hirschfeldt \cite[Corollary 8.12.2]{downey-hirschfeldt-2010}),
so it contains a 1-random set; in particular it contains a computably independent set $A$. In this section we will show that we can use such sets to obtain factors with theory $\mathrm{IPC}$ below $\B$, by relativising Theorem \ref{thm-main}.

However, we first show that we can relativise Theorem \ref{thm-pow-embed} below $\B$.

\begin{thm}\label{thm-pow-embed-rel}
Let $\B$ be a mass problem, let $\E$ be a mass problem such that $\E \not\geq_\M \B$ and let $\D = \E \to_\M \B$. Let $\A \subseteq \omega^\omega$ be a mass problem which is downwards closed under Turing reducibility such that $\E \subseteq \overline{\A}$. Let $(f_i)_{i \in I}$ be an antichain in $\A$. Then the map $\beta$ given by $\beta(X) = (\overline{\A} \cup \{g \mid \exists i \in X (g \geq_T f_i)\}) \otimes \D$ is an upper implicative semilattice embedding of $(\mathcal{P}(I),\supseteq)$ into $[\beta(I),\beta(\emptyset)]_\M$ with range canonical in $[\beta(I),\beta(\emptyset)]_\M$.
\end{thm}
\begin{proof}
First, note that $\E \not\geq_\M \D$, since if $\E \geq_\M \D$ then
\[\E \equiv_\M \E \oplus \D = \E \oplus (\E \to \B) \geq_\M \B,\]
a contradiction.

As in the proof of Theorem \ref{thm-pow-embed}, if $X \subseteq I$ we will denote by $C(X)$ the set $C(\{f_i \mid i \in X\})$. By Theorem \ref{thm-pow-embed}, the function $\alpha: \mathcal{P}(I) \to \mathrm{\M / \overline{\A}}$ given by $\alpha(X) = \overline{\A} \cup C(X)$ is an upper implicative semilattice embedding of $(\mathcal{P}(I),\supseteq)$ into $\Big[\overline{\A} \cup C(I),\overline{\A}\Big]_\M$. Note that $\E \subseteq \overline{\A}$ and therefore $\E \subseteq \alpha(X)$ for every $X \subseteq I$.

Now let $\beta: \mathcal{P}(I) \to \mathrm{\M / \overline{\A}}$ be the function given by $\beta(X) = \alpha(X) \otimes \D$. Then the range of $\beta$ is certainly contained in $[\beta(I),\beta(\emptyset)]_\M$. We prove that $\beta$ is in fact an upper implicative semilattice embedding into $[\beta(I),\beta(\emptyset)]_\M$ with canonical range.

\begin{itemize}
\item $\beta$ is injective: assume $\beta(X) \leq_\M \beta(Y)$. Thus, we have $\alpha(X) \otimes \D \leq_\M \alpha(Y) \otimes \D$. In particular we have that $\alpha(X) \otimes \D \leq_\M \alpha(Y)$, say via $\Phi_n$. We claim: $\Phi_n(\alpha(Y)) \subseteq 0 \conc \alpha(X)$.

Namely, assume towards a contradiction that $\Phi_n(f) \in 1 \conc \D$ for some $f \in \alpha(Y)$. Determine $\sigma \subseteq f$ such that $\Phi_n(\sigma)(0) = 1$. As noted above we have that $\E \subseteq \alpha(Y)$, and since $\alpha(Y)$ is Muchnik we therefore see that $\sigma \conc \E \subseteq \alpha(Y)$. However, then we can reduce $\E$ to $1 \conc \D$ by sending $g \in \E$ to $\Phi_n(\sigma \conc g)$, and therefore $\E \geq_\M \D$, a contradiction.

Thus, $\alpha(X) \leq_\M \alpha(Y)$, and since $\alpha$ is an upper implicative semilattice embedding this tells us that $X \supseteq Y$.

\item $\beta$ preserves joins: we have
\begin{align*}
\beta(X \oplus Y) &= \alpha(X \oplus Y) \otimes \D \equiv_\M (\alpha(X) \oplus \alpha(Y)) \otimes \D\\
&\equiv_\M (\alpha(X) \otimes \D) \oplus (\alpha(Y) \otimes \D) = \beta(X) \oplus \beta(Y).
\end{align*}

\item $\beta$ preserves implications: we have
\begin{align*}
\beta(X) &\to_{[\beta(I),\beta(\emptyset)]_\M} \beta(Y)\\
&= ((\alpha(X) \otimes \D) \to_\M (\alpha(Y) \otimes \D)) \oplus \beta(I)\\
&\equiv_\M ((\alpha(X) \to_\M (\alpha(Y) \otimes \D)) \oplus (\D \to_\M (\alpha(Y) \otimes \D))) \oplus \beta(I)\\
&\equiv_\M ((\alpha(X) \to_\M (\alpha(Y) \otimes \D)) \oplus \omega^\omega) \oplus \beta(I)\\
&\equiv_\M (\alpha(X) \to_\M (\alpha(Y) \otimes \D)) \oplus \beta(I).\\
\intertext{Next, using Proposition \ref{prop-muchnik-can} we see:}
&\equiv_\M ((\alpha(X) \to_\M \alpha(Y)) \otimes (\alpha(X) \to_\M \D)) \oplus \beta(I)\\
&= ((\alpha(X) \to_\M \alpha(Y)) \otimes (\alpha(X) \to_\M (\E \to_\M \B))) \oplus \beta(I)\\
&\equiv_\M ((\alpha(X) \to_\M \alpha(Y)) \otimes ((\alpha(X) \oplus \E) \to_\M \B)) \oplus \beta(I).\\
\intertext{As noted above, we have $\E \subseteq \alpha(X)$, and therefore:}
&\equiv_\M ((\alpha(X) \to_\M \alpha(Y)) \otimes (\E \to_\M \B)) \oplus \beta(I)\\
&\equiv_\M ((\alpha(X) \to_\M \alpha(Y)) \otimes \D) \oplus (\alpha(I) \otimes \D).\\
&\equiv_\M ((\alpha(X) \to_\M \alpha(Y)) \oplus \alpha(I)) \otimes \D\\
&= \left(\alpha(X) \to_{[\alpha(I),\alpha(\emptyset)]_\M} \alpha(Y)\right) \otimes \D\\
&= \alpha(X \to_{\mathcal{P}(I)} Y) \otimes \D\\
&= \beta(X \to_{\mathcal{P}(I)} Y).
\end{align*}

\item $\beta$ has canonical range:
\begin{enumerate}[\rm (i)]
\item Let $X \subseteq I$, we show that that $\beta(X)$ is meet-irreducible in $[\beta(I),\beta(\emptyset)]_\M$. Indeed, let $\C_0,\C_1 \leq_\M \beta(\emptyset)$ be such that $\C_0 \otimes \C_1 \leq_\M \alpha(X) \otimes \D$. Then $\C_0 \otimes \C_1 \leq_\M \alpha(X)$, and since $\alpha(X)$ is Muchnik, we see from Proposition \ref{prop-muchnik-can} that $\C_0 \leq_\M \alpha(X)$ or $\C_1 \leq_\M \alpha(X)$. Since $\C_0,\C_1 \leq_\M \beta(\emptyset) \leq_\M \D$ this shows that in fact $\C_0 \leq_\M \beta(X)$ or $\C_1 \leq_M \beta(X)$.
\item The range of $\beta$ is clearly closed under implication and joins.
\item Let $X \subseteq \omega$ and let $\C_0,\C_1 \in [\beta(I),\beta(\emptyset)]_\M$.
Then we have:
\begin{align*}
\beta(X) &\to_{[\beta(I),\beta(\emptyset)]_\M} (\C_0 \otimes \C_1)\\
&=(\alpha(X) \otimes \D) \to_{[\beta(I),\beta(\emptyset)]_\M} (\C_0 \otimes \C_1)\\
&=((\alpha(X) \otimes \D) \to_\M (\C_0 \otimes \C_1)) \oplus \beta(I)\\
&\equiv_\M (\alpha(X) \to_\M (\C_0 \otimes \C_1)) \oplus \beta(I),\\
\intertext{because $\C_0$ and $\C_1$ are below $\beta(\emptyset)$ and hence below $\D$. Since $\alpha(X)$ is Muchnik, we now see from Proposition \ref{prop-muchnik-can}:}
&= ((\alpha(X) \to_\M \C_0) \otimes (\alpha(X) \to_\M \C_1)) \oplus \beta(I)\\
&\equiv_\M (\beta(X) \to_{[\beta(I),\beta(\emptyset)]_\M} \C_0) \otimes (\beta(X) \to_{[\beta(I),\beta(\emptyset)]_\M} \C_1).\qedhere
\end{align*}
\end{enumerate}
\end{itemize}
\end{proof}

We can now prove there is a principal factor of the Medvedev lattice with theory IPC below a given $\B > 0'$.

\begin{thm}\label{thm-factor-rel}
Let $\B$ be a mass problem, let $\E$ be a countable mass problem such that $\E \not\geq_\M \B$ and let $\D = \E \to \B$ (so, $\D \leq_\M \B$). Let $A$ be a computably independent set such that for all $f \in \E$ we have $A \not\geq_T f$. Then
\[\mathrm{Th}\left(\M / \left(\left\{i \conc g \mid g \geq_T A^{[i]} \text{ or } g \in C(\E)\right\} \otimes \D\right)\right) = \mathrm{IPC}.\]
\end{thm}
\begin{proof}
The proof largely mirrors that of Theorem \ref{thm-main}. Let $\E = \{f_0,f_1,\dots\}$, let $E_i$ be the graph of $f_i$ and let $U$ be such that $U^{[0]} = A$ and $U^{[i+1]} = E_i$. Then $A,E_0,E_1,\dots$ is uniformly computable in $U$.

We need to make a slight modification to Lemma \ref{lem-extend}: we not only want $A \oplus B$ to be computably independent, but we also need to make sure that $A \oplus B \not\geq_T f$ for every $f \in \E$. This modification is straightforward and we omit the details.
The requirements on $D_i$ are slightly different: 
we now want for every $1 \leq i \leq k$ that $D_i \geq_T \bigoplus_{1 \leq j \leq k, i \in X_j} A^{[j]} \oplus B^{[i]}$, that $D_i' \leq_T (U \oplus B)'$, that $D_i \oplus A^{[j]} \geq_T (U \oplus B)'$ for every $j \in \{1 \leq j \leq k \mid i \not\in X_j\} \cup \{k+1,k+2,\dots\}$, that $D_i \oplus B^{[j]} \geq_T (U \oplus B)'$ for every $j \not= i$ and that $D_i \oplus E_j \geq_T (U \oplus B)'$ for all $j \in \omega$; this is still possible by Theorem \ref{thm-splitting-ext}. We change the definition of $\A$ into $\A = \omega^\omega \setminus C(\{(U \oplus B)'\} \cup \E)$. Then we still have $D_i \in \A$, because $D_i \geq_T f_j$ would imply that $D_i \geq_T D_i \oplus E_j \geq_T (U \oplus B)'$, a contradiction. Finally, replace $\alpha$ with the $\beta$ of Theorem \ref{thm-pow-embed-rel} and change \eqref{thm-main-eqn1} into
\begin{align*}
\left(\Big(\overline{\A} \cup C(\{D_1,\dots,D_n\})\Big) \otimes \D\right) &\oplus \left(\left\{i \conc g \mid g \geq_T A^{[i]} \text{ or } g \in C(\E)\right\} \otimes \D\right)\\
&\equiv_\M \beta(X_1) \otimes \dots \otimes \beta(X_k).
\end{align*}
Then the whole proof of Theorem \ref{thm-main} goes through.
\end{proof}

In particular, this allows us to give a positive answer to the question mentioned at the beginning of this section.

\begin{cor}\label{cor-jan}
Let $\B >_\M 0'$. Then $\mathrm{Th}(\M / \B) \subseteq \mathrm{Jan}$.
\end{cor}
\begin{proof}
Since an intermediate logic is contained in $\mathrm{Jan}$ if and only if its positive fragment coincides with $\mathrm{IPC}$ (see Jankov \cite{jankov-1968}),
we need to show that, denoting the positive fragment by $^{+}$, we have that $\mathrm{Th}^{+}(\M / \B) \subseteq \mathrm{IPC}^{+}$. By Theorem \ref{thm-factor-rel} there exists a $\C \leq_\M \B$ such that $\mathrm{Th}(\M / \C) = \mathrm{IPC}$. Then $\M / \C$ is a subalgebra of $\M / \B$, except for the fact that the top element is not necessarily preserved. However, it can be directly verified that for any two Brouwer algebras $\BC$ and $\BB$ for which $\BC$ is a $(\oplus,\otimes,\to,0)$-subalgebra of $\BB$ we have for all positive formulas $\phi(x_1,\dots,x_n)$ and all elements $b_1,\dots,b_n \in \BB$ that the interpretation of $\phi$ at $b_1,\dots,b_n$ is the same in both $\BC$ and $\BB$. Since we can refute every positive formula $\phi$ which is not in $\mathrm{IPC}^+$ in $\M / \C$, we can therefore refute it in $\M / \B$ using the same valuation. In other words, $\mathrm{Th}^{+}(\M / \B) \subseteq \mathrm{Th}^{+}(\M / \C) = \mathrm{IPC}^{+}$, as desired.
\end{proof}

\section*{Acknowledgements}

The author thanks Sebastiaan Terwijn for helpful discussions on the subject. Furthermore, the author thanks the anonymous referees for their many useful comments.

\providecommand{\bysame}{\leavevmode\hbox to3em{\hrulefill}\thinspace}
\providecommand{\MR}{\relax\ifhmode\unskip\space\fi MR }
\providecommand{\MRhref}[2]{%
  \href{http://www.ams.org/mathscinet-getitem?mr=#1}{#2}
}
\providecommand{\href}[2]{#2}

\end{document}